\documentclass[reqno]{amsart}

\usepackage[latin1]{inputenc}
\usepackage{amssymb}
\usepackage{graphicx}
\usepackage{amscd}
\usepackage[hidelinks]{hyperref}
\usepackage{color}
\usepackage{float}
\usepackage{graphics,amsmath,amssymb}
\usepackage{amsthm}
\usepackage{amsfonts}
\usepackage{latexsym}
\usepackage{epsf}
\usepackage{enumerate}
\usepackage{xifthen}
\usepackage{mathrsfs}
\usepackage{dsfont}
\usepackage{makecell}
\usepackage[FIGTOPCAP]{subfigure}
\usepackage{amsmath}
\allowdisplaybreaks[4]
\usepackage{listings}
\usepackage{etoolbox}
\usepackage{fancyhdr}

 \newtheoremstyle{mytheorem}
 {3pt}
 {3pt}
 {\slshape}
 {}
 {\bfseries}
 {.}
 { }
 {}

\numberwithin{equation}{section}

\theoremstyle{theorem}
\newtheorem{theorem}{Theorem}[section]

\newtheorem{lemma}[theorem]{Lemma}

\theoremstyle{definition}

\newtheorem{remark}{Remark}[section]

\newcommand{\doi}[1]{\href{http://dx.doi.org/#1}{doi: #1}}

\newcommand{\Keywords}[1]{\ifthenelse{\isempty{#1}}{}{\smallskip \smallskip \noindent \textbf{Keywords}. #1}}
\newcommand{\MSC}[2][2010]{\ifthenelse{\isempty{#2}}{}{\smallskip \smallskip \noindent \textbf{#1MSC}. #2}}
\newcommand{\abstractnote}[1]{\ifthenelse{\isempty{#1}}{}{\smallskip \smallskip \noindent \textsuperscript{\dag}#1}}

\makeatletter
\def\specialsection{\@startsection{section}{1}%
  \z@{\linespacing\@plus\linespacing}{.5\linespacing}%
  {\normalfont}}
\def\section{\@startsection{section}{1}%
  \z@{.7\linespacing\@plus\linespacing}{.5\linespacing}%
  {\normalfont\scshape}}
\patchcmd{\@settitle}{\uppercasenonmath\@title}{\Large\boldmath}{}{}
\patchcmd{\@settitle}{\begin{center}}{\begin{flushleft}}{}{}
\patchcmd{\@settitle}{\end{center}}{\end{flushleft}}{}{}
\patchcmd{\@setauthors}{\MakeUppercase}{\normalsize}{}{}
\patchcmd{\@setauthors}{\centering}{\raggedright}{}{}
\patchcmd{\section}{\scshape}{\large\bfseries\boldmath}{}{}
\patchcmd{\subsection}{\bfseries}{\bfseries\boldmath}{}{}
\renewcommand{\@secnumfont}{\bfseries}
\patchcmd{\@startsection}{\@afterindenttrue}{\@afterindentfalse}{}{}
\patchcmd{\abstract}{\leftmargin3pc}{\leftmargin1pc}{}{}

\def\maketitle{\par
  \@topnum\z@ 
  \@setcopyright
  \thispagestyle{empty}
  \ifx\@empty\shortauthors \let\shortauthors\shorttitle
  \else \andify\shortauthors
  \fi
  \@maketitle@hook
  \begingroup
  \@maketitle
  \toks@\@xp{\shortauthors}\@temptokena\@xp{\shorttitle}%
  \toks4{\def\\{ \ignorespaces}}
  \edef\@tempa{%
    \@nx\markboth{\the\toks4
      \@nx\MakeUppercase{\the\toks@}}{\the\@temptokena}}%
  \@tempa
  \endgroup
  \c@footnote\z@
  \@cleartopmattertags
}
\makeatother

\title{Congruences and recursions for the cubic partition} 

\author[S. Chern]{Shane Chern}
\address[S. Chern]{School of Mathematical Sciences, Zhejiang University, Hangzhou, 310027, China}
\email{shanechern@zju.edu.cn; chenxiaohang92@gmail.com}

\author[M. G. Dastidar]{Manosij Ghosh Dastidar}
\address[M. G. Dastidar]{Department of Mathematical Sciences, Pondicherry University, R. V. Nagar, Kalapet, Puducherry, PIN-605014, India}
\email{gdmanosij@gmail.com}

\date{}

\begin{document}

{\footnotesize\noindent To appear in \textit{Ramanujan J.}\\
\doi{10.1007/s11139-016-9852-7}}

\bigskip \bigskip

\maketitle

\begin{abstract}
Let $p_2(n)$ denote the number of cubic partitions. In this paper, we shall present two new congruences modulo $11$ for $p_2(n)$. We also provide an elementary alternative proof of a congruence established by Chan. Furthermore, we will establish a recursion for $p_2(n)$, which is a special case of a broader class of recursions.


\MSC{Primary 11P83; Secondary 05A17.}
\end{abstract}

\section{Introduction}

A partition of a natural number $n$ is a nonincreasing sequence of positive integers whose sum equals $n$. Let $p(n)$ be the number of partitions. Among Ramanujan's discoveries, the following identity:
$$\sum_{n\ge 0}p(5n+4)q^n=5\frac{(q^5;q^5)_\infty^5}{(q;q)_\infty^6},$$
is regarded as his ``Most Beautiful Identity'' by both Hardy and MacMahon; see \cite[p. xxxv]{Ram2000}. Here as usual we denote
$$(a;q)_\infty = \prod_{n\ge 0}(1-aq^n).$$
This identity immediately leads to the following famous congruence:
$$p(5n+4)\equiv 0\pmod{5}.$$
Ramanujan also discovered two congruences with different moduli, namely
\begin{align*}
p(7n+5)&\equiv 0\pmod{7},\\
p(11n+6)&\equiv 0\pmod{11}.
\end{align*}

Motivated by Ramanujan's result, Chan \cite{Chan2010} introduced the notion of cubic partition of nonnegative integers. Let $p_2(n)$ be the number of such partitions. Its generating function is given by
\begin{equation}\label{eq:a}
\sum_{n\ge 0}p_2(n)q^n = \frac{1}{(q;q)_\infty (q^2;q^2)_\infty},\quad |q|<1.
\end{equation}
From an identity on the Ramanujan's cubic continued fraction, Chan established the following elegant identity:
\begin{equation}
\sum_{n\ge 0} p_2(3n+2)q^n=3\frac{(q^3;q^3)_\infty^3(q^6;q^6)_\infty^3}{(q;q)_\infty^4(q^2;q^2)_\infty^4},
\end{equation}
which immediately implies
\begin{equation}\label{eq:mod3}
p_2(3n+2)\equiv 0 \pmod{3}.
\end{equation}
Later on, many authors studied other Ramanujan-like congruences for $p_2(n)$. For example, Chen and Lin \cite{CL2009} found four new congruences modulo $7$ by using modular forms, whereas Xiong \cite{Xiong2011} established sets of congruences modulo powers of $5$.

We will present two new congruences modulo $11$ next in Sect.~\ref{sec:02}. Then in Sect.~\ref{sec:03}, we will provide an elementary alternative proof of \eqref{eq:mod3}. At last, we will establish a recursion for $p_2(n)$, which is a special case of a broader class of recursions.

\section{New congruences modulo 11 for $\boldsymbol{p_2(n)}$}\label{sec:02}

In this section, we shall present two new congruences modulo $11$ for $p_2(n)$. Unlike previous congruences modulo $5$ or $7$, the two congruences are of the type $p_2(297n+t)$, with $297=3^3\times 11$ not being the square of $11$. Our result is

\begin{theorem}\label{th:11}
For any nonnegative integer $n$,
\begin{equation}\label{eq:mod11}
p_2(297n+t)\equiv 0\pmod{11},
\end{equation}
where $t=62$ and $161$.
\end{theorem}

To prove the two congruences, we need to use a result of Radu and Sellers \cite[Lemma 2.4]{RS2011}, which can be tracked back to \cite[Lemma 4.5]{Radu2009}. Before introducing the result of Radu and Sellers, we will briefly interpret some notations.

Let $\Gamma:=SL_2(\mathbb{Z})$. For a positive integer $N$, the congruence subgroup $\Gamma_0(N)$ of level $N$ is defined by
$$\Gamma_0(N)=\left\{\left.\begin{pmatrix}a & b\\c & d\end{pmatrix}\ \right|\ c\equiv 0\pmod{N}\right\}.$$
It is known that
$$[\Gamma:\Gamma_0(N)]=N\prod_{p\mid N}(1+p^{-1}).$$
Moreover, we write
$$\Gamma_\infty=\left\{\left.\begin{pmatrix}1 & h \\0 & 1 \end{pmatrix}\ \right|\ h\in\mathbb{Z}\right\}.$$

For a positive integer $M$, let $R(M)=\{r:r=(r_{\delta_1},\ldots,r_{\delta_k})\}$ be the set of integer sequences indexed by the positive divisors $1=\delta_1<\cdots<\delta_k=M$ of $M$. Let $m$ be a positive integer and $[s]_m$ the set of all elements congruent to $s$ modulo $m$. Let $\mathbb{Z}_m^*$ denote the set of all invertible elements in $\mathbb{Z}_m$, and $\mathbb{S}_m$ denote the set of all squares in $\mathbb{Z}_m^*$. For $t\in\{0,\ldots,m-1\}$, let $\overline{\odot}_r$ be the map $\mathbb{S}_{24m}\times\{0,\ldots,m-1\}$$\to$$\{0,\ldots,m-1\}$ with
$$([s]_{24m},t)\mapsto [s]_{24m}\overline{\odot}_r t\equiv ts+\frac{s-1}{24}\sum_{\delta\mid M}\delta r_{\delta}\pmod{m},$$
and write $P_{m,r}(t)=\{[s]_{24m}\overline{\odot}_r t\ |\ [s]_{24m}\in\mathbb{S}_{24m}\}$.

Denote by $\Delta^*$ the set of tuples $(m,M,N,t,r=(r_\delta))$ satisfying conditions given in \cite[p. 2255]{RS2011}. Let $\kappa=\kappa(m)=\gcd(m^2-1,24)$. We set
$$p_{m,r}(\gamma)=\min_{\lambda\in\{0,\ldots,m-1\}}\frac{1}{24}\sum_{\delta\mid M}r_\delta\frac{\gcd^2(\delta(a+\kappa\lambda c),mc)}{\delta m},$$
and
$$p_{r'}^*(\gamma)=\frac{1}{24}\sum_{\delta\mid N} \frac{r'_\delta\gcd^2(\delta,c)}{\delta},$$
where $\gamma=\begin{pmatrix}a & b \\c & d \end{pmatrix}$, $r\in R(M)$, and $r'\in R(N)$.

Let
$$f_r(q):=\prod_{\delta\mid M}(q^{\delta};q^{\delta})_{\infty}^{r_\delta}=\sum_{n\ge 0}c_r(n)q^n$$
for some $r\in R(M)$. The lemma of Radu and Sellers is given as follows.

\begin{lemma}\label{le:01}
Let $u$ be a positive integer, $(m,M,N,t,r=(r_\delta))\in\Delta^*$, $r'=(r'_\delta)\in R(N)$, $n$ be the number of double cosets in $\Gamma_0(N)\backslash\Gamma/\Gamma_\infty$ and $\{\gamma_1,\ldots,\gamma_n\}$ $\subset\Gamma$ be a complete set of representatives of the double coset $\Gamma_0(N)\backslash\Gamma/\Gamma_\infty$. Assume that $p_{m,r}(\gamma_i)+p_{r'}^*(\gamma_i)\ge 0$ for all $i=1,\ldots,n$. Let $t_{\min} := \min_{t'\in P_{m,r}(t)}t'$ and
$$v:=\frac{1}{24}\left(\left(\sum_{\delta\mid M}r_\delta+\sum_{\delta\mid N}r'_\delta\right)[\Gamma:\Gamma_0(N)]-\sum_{\delta\mid N}\delta r'_\delta\right)-\frac{1}{24m}\sum_{\delta\mid M}\delta r_\delta-\frac{t_{\min}}{m}.$$
Then if
$$\sum_{n=0}^{\lfloor v \rfloor}c_r(mn+t')q^n\equiv 0 \pmod{u},$$
for all $t'\in P_{m,r}(t)$, then
$$\sum_{n\ge 0}c_r(mn+t')q^n\equiv 0 \pmod{u},$$
for all $t'\in P_{m,r}(t)$.
\end{lemma}

\begin{proof}[Proof of Theorem \ref{th:11}]
By the binomial theorem and \eqref{eq:a}, one readily sees that
\begin{equation}\label{eq:a11}
\sum_{n\ge 0}p_2(n)q^n\equiv\frac{(q;q)_\infty^{10}}{(q^2;q^2)_\infty(q^{11};q^{11})_\infty}=:\sum_{n\ge 0}g_{2,11}(n)q^n \pmod{11}.
\end{equation}

We first consider the case of $p_2(297n+62)$, and set
$$(m,M,N,t,r=(r_1,r_2,r_{11},r_{22}))=(297,22,66,62,(10,-1,-1,0))\in\Delta^*.$$
By the definition of $P_{m,r}(t)$, we obtain
$$P_{m,r}(t)=\left\{t'\ |\ t'\equiv ts-(s-1)/8\ (\bmod\ m),0\le t'\le m-1,[s]_{24m}\in\mathbb{S}_{24m}\right\}.$$
We readily verify that $P_{m,r}(t)=\{62\}$, and set
$$r'=(r'_1,r'_2,r'_3,r'_6,r'_{11},r'_{22},r'_{33},r'_{66})=(4, 2, 0, 0, 0, 1, 0, 0).$$
Now let
$$\gamma_\delta=\begin{pmatrix}
1 & 0 \\
\delta & 1 
\end{pmatrix}.$$
It follows by \cite[Lemma 2.6]{RS2011} that $\{\gamma_\delta:\delta\mid N\}$ contains a complete set of representatives of the double coset $\Gamma_0(N)\backslash\Gamma/\Gamma_\infty$. Since all these constants satisfy the assumption of Lemma \ref{le:01}, we obtain the upper bound $\lfloor v \rfloor=88$. Through a similar process, one may see the the upper bound $\lfloor v \rfloor$ for the $p_2(297n+161)$ case is also $88$. By Lemma \ref{le:01}, we only need to verify terms up to this bound.

Now we will complete our proof with the help of \textit{Mathematica}. We first note that
$$p_2(n)=\sum_{\substack{i+2j=n\\i,j\ge 0}}p(i)p(j).$$
Note also that $p(n)$ is computable by the \textit{Mathematica} function \texttt{PartitionsP}. One readily verifies that \eqref{eq:mod11} holds for both $t=62$ and $161$ when $n\le 88$. This ends the proof of Theorem \ref{th:11}.
\end{proof}

\begin{remark}
It is still natural to ask if there are elementary proofs of the two congruences. Considering the difficulty of finding $11$-dissection formulas for some $q$-series products, we leave this as an open problem. 
\end{remark}

\section{An elementary alternative proof for Chan's congruence}\label{sec:03}

Although we fail to give an elementary proof for our Theorem \ref{th:11}, we do find an elementary alternative proof for Chan's congruence \eqref{eq:mod3}.

Note that
\begin{align*}
\sum_{n\ge 0}p_2(n)q^n &= \frac{1}{(q;q)_\infty (q^2;q^2)_\infty}=\prod_{n\ge 1}\frac{1}{1-q^n}\prod_{n\ge 1}\frac{1}{1-q^{2n}}\\
&= \prod_{n\ge 1}\frac{1-q^n}{1+q^n}\left(\prod_{n\ge 1}\frac{1}{1-q^n}\right)^3.
\end{align*}
It is well known that
$$\sum_{n\ge 0}s(n^2)q^{n^2}:=1+2\sum_{n\ge 1}(-q)^{n^2}=\prod_{n\ge 1}\frac{1-q^n}{1+q^n};$$
see \cite[Chapter 16, Entry 22(i)]{Ber1991} and \cite[Chapter 16, Eq. (22.4)]{Ber1991}.
We therefore have
\begin{equation}
p_2(n)=\sum_{\substack{m^2+i+j+k=n\\m,i,j,k\ge 0}}s(m^2)p(i)p(j)p(k).
\end{equation}

Since $3n+2-m^2\equiv 1$ or $2$ (mod $3$), at least two of $i$, $j$, $k$, the solution to
\begin{equation}\label{eq:ijk}
i+j+k=3n+2-m^2,
\end{equation}
are distinct. If the pairwise distinct triple $(i,j,k)$ is a solution to \eqref{eq:ijk}, then any permutation of $(i,j,k)$ [viz., $(j,k,i)$, etc.] is a solution to \eqref{eq:ijk}. If $i=j\ne k$, then $(i,k,i)$ and $(k,i,i)$ are also solutions to \eqref{eq:ijk}. We therefore obtain
\begin{align}
p_2(3n+2)=&6\sum_{\substack{m^2+i+j+k=3n+2\\m\ge 0,i>j>k\ge 0}}s(m^2)p(i)p(j)p(k)\notag\\
&+3\sum_{\substack{m^2+i+j+k=3n+2\\m,i,j,k\ge 0,\ i=j\ne k}}s(m^2)p(i)p(j)p(k).
\end{align}
This leads to
\begin{theorem}[Chan]\label{th:Chan}
For any nonnegative integer $n$,
\begin{equation}\label{eq:mod3b}
p_2(3n+2)\equiv 0\pmod{3}.
\end{equation}
\end{theorem}

\section{Recursion for the cubic partition}

We know that the popular recursion of $p(n)$ links partitions to the divisor function. In this section, we wish to show that a similar recursion applies to $p_2(n)$. Actually, this is a special case of recursions for two-color partitions where one of the colors appears only in parts that are multiples of $k$. Let $p_k(n)$ denote the number of such partitions. According to \cite{ABD2015}, its generating function is
\begin{equation}\label{eq:pk}
\sum_{n\ge 0}p_k(n)q^n=\frac{1}{(q;q)_\infty(q^k;q^k)_\infty},\quad |q|<1.
\end{equation}
For properties of $p_k(n)$, the reader may refer to \cite{ABD2015,Chern2015}.

Recall that Ford's recursion for $p(n)$ is as follows:
\begin{equation}
p(n)=\frac{1}{n}\sum_{m=1}^n \sigma(m)p(n-m),
\end{equation}
where $\smash{\sigma(n)=\sum_{d\mid n}d}$; see \cite{Ford1931}. The reader may also refer to the papers of Erd\"os \cite{Erdos1942} and Andrews and Deutsch \cite{AD2015} for other interesting aspects of this identity. Let $\sigma^{(k)}(n)$ denote the sum of $k$-labeled divisors of $n$, that is, the multiples of $k$ have two labels. For example, $\sigma^{(2)}(4)=1+2_1+2_2+4_1+4_2=13$. Our result is
\begin{theorem}\label{th:pk}
For any nonnegative integer $n$,
\begin{equation}\label{eq:pkre}
p_k(n)=\frac{1}{n}\sum_{m=1}^n \sigma^{(k)}(m)p_k(n-m).
\end{equation}
\end{theorem}

\begin{proof}
Taking
$$F(q)=\sum_{n\ge 0}p_k(n)q^n,$$
then
$$qF'(q)=\sum_{n\ge 1}np_k(n)q^n.$$
Let $G(q)=1/F(q)=\prod_{n\ge 1}\{(1-q^n)(1-q^{kn})\}$, we have
\begin{equation}\label{eq:proof4.1}
qF'(q)=-q\frac{G'(q)}{G(q)^2}=-q\frac{G'(q)}{G(q)}F(q).
\end{equation}
Note that
\begin{align*}
-q\frac{G'(q)}{G(q)}&=-q\left(\log G(q)\right)'\\
&=\sum_{n\ge 1}\frac{nq^n}{1-q^n}+\sum_{n\ge 1}\frac{knq^{kn}}{1-q^{kn}}.
\end{align*}
It is also known that
$$\frac{nq^n}{1-q^n}=\sum_{k\ge 1}nq^{kn},$$
we therefore obtain
$$-q\frac{G'(q)}{G(q)}=\sum_{n\ge 1}\left(\sum_{d\mid n}d+\sum_{\substack{d\mid n\\k\mid d}}d\right)q^n=\sum_{n\ge 1}\sigma^{(k)}(n)q^n.$$
Combining it with \eqref{eq:proof4.1}, one immediately sees that
$$np_k(n)=\sum_{m=1}^n \sigma^{(k)}(m)p_k(n-m).$$
This ends the proof of Theorem \ref{th:pk}.
\end{proof}

\bibliographystyle{amsplain}

\begin{thebibliography}{99}

\bibitem{ABD2015}
Z. Ahmed, N. D. Baruah, and M. G. Dastidar, New congruences modulo $5$ for the number of $2$-color partitions, \textit{J. Number Theory} \textbf{157} (2015), 184--198.

\bibitem{AD2015}
G. E. Andrews and E. Deutsch, A note on a method of Erd\"os and the Stanley-Elder theorems, \textit{Integers} \textbf{16} (2016), Paper No. A24, 5 pp.

\bibitem{Ber1991}
B. C. Berndt, \textit{Ramanujan's notebooks. Part III}, Springer-Verlag, New York, 1991. xiv+510 pp.

\bibitem{Chan2010}
H.-C. Chan, Ramanujan's cubic continued fraction and an analog of his ``most beautiful identity'', \textit{Int. J. Number Theory} \textbf{6} (2010), no. 3, 673--680.

\bibitem{CL2009}
W. Y. C. Chen and B. L. S. Lin, Congruences for the number of cubic partitions derived from modular
forms, \textit{Preprint}, arXiv:0910.1263, 15 pp.

\bibitem{Chern2015}
S. Chern, New congruences for $2$-color partitions, \textit{J. Number Theory} \textbf{163} (2016), 474--481.

\bibitem{Erdos1942}
P. Erd\"os, On an elementary proof of some asymptotic formulas in the theory of partitions, \textit{Ann. of Math. (2)} \textbf{43} (1942), 437--450.

\bibitem{Ford1931}
W. B. Ford, Two theorems on the partitions of numbers, \textit{Amer. Math. Monthly} \textbf{38} (1931), no. 4, 183--184.

\bibitem{Radu2009}
S. Radu, An algorithmic approach to Ramanujan's congruences, \textit{Ramanujan J.} \textbf{20} (2009), no. 2, 215--251.

\bibitem{RS2011}
S. Radu and J. A. Sellers, Congruence properties modulo $5$ and $7$ for the $\mathrm{pod}$ function, \textit{Int. J. Number Theory} \textbf{7} (2011), no. 8, 2249--2259. 

\bibitem{Ram2000}
S. Ramanujan, \textit{Collected papers of Srinivasa Ramanujan}, AMS Chelsea Publishing, Providence, RI, 2000. xxxviii+426 pp.

\bibitem{Xiong2011}
X. H. Xiong, The number of cubic partitions modulo powers of $5$ (Chinese), \textit{Sci. Sin. Math.} \textbf{41} (2011), no. 1, 1--15.

\end{thebibliography}

\end{document}